\documentclass[11pt]{article}

\usepackage{amsmath,amssymb}
\usepackage{latexsym}
\usepackage{graphicx}
\usepackage{bm}

\newtheorem{thm}{Theorem}

\newenvironment{proof}{\begin{trivlist}
                       \item[]{\bf Proof.}
                       \hspace{0cm}}{\hfill $\Box$
                       \end{trivlist}}

\begin{document}
\title{Symmetry problems 2}

\author{N. S. Hoang$\dag$\footnotemark[1] \quad and \quad
	  A. G. Ramm$\dag$\footnotemark[3] \\
\\
$\dag$Mathematics Department, Kansas State University,\\
Manhattan, KS 66506-2602, USA
}

\renewcommand{\thefootnote}{\fnsymbol{footnote}}
\footnotetext[1]{Email: nguyenhs@math.ksu.edu}
\footnotetext[3]{Corresponding author. Email: ramm@math.ksu.edu}

\date{}
\maketitle
\begin{abstract}
 Some symmetry problems are formulated and solved. New simple proofs are given for the earlier studied
symmetry problems.
\end{abstract}



\section{Introduction}

Symmetry problems are of interest both theoretically and in applications.

A well-known, and still unsolved, symmetry problem is the Pompeiu problem 
(see \cite{R363}, \cite{R382}). It consists of proving the following: 

{\it If 
$D\subset \mathbb{R}^n,\, n\ge 2$
is homeomorphic to a ball, and the boundary $S$ of $D$ is sufficiently 
smooth, ($S\in C^{1,\lambda},\, \lambda>0$ is sufficient) and if the problem
\begin{equation}
 (\nabla^2 + k^2)u = 0\quad \text{in}\quad D,\qquad u\big{|}_S = c,\quad u_N\big{|}_S=0,
 \quad k^2 = const>0,
\end{equation}
has a solution, then $S$ is a sphere.}

A similar problem ({\it Schiffer's conjecture}) is also unsolved: 

{\it If the 
problem
\begin{equation}
 (\nabla^2 + k^2)u = 0\quad \text{in}\quad D,\qquad u\big{|}_S = 0,\quad u_N\big{|}_S=c\not=0,
 \quad k^2 = const>0,
\end{equation}
has a solution, then $S$ is a sphere.}

In \cite{R512} it is proved that if
\begin{equation}
\int_D \frac{dy}{4\pi |x-y|} = \frac{c}{|x|},\qquad \forall x\in B'_R, 
\quad c = const>0,
\end{equation}
then $D$ is a ball. Here and below we assume that $D\subset \mathbb{R}^3$ 
is a bounded domain
homeomorphic to a ball, with a sufficiently smooth boundary $S$ ($S$ is 
Lipschitz suffices),
$B_R = \{ x: |x|\le R \}$, $B_R\supset D$. By $\mathcal{H}$ we 
denote the set of all harmonic functions
in a domain which contains $D$. By $|D|$ and $|S|$ we denote the volume of 
$D$ and the surface 
area of $S$, respectively.

Our goal is to give a simple proof of the three symmetry-type results, 
formulated in Theorem~\ref{thm1} in Section 2. 

In \cite{S} the following result is obtained:

{\it If
\begin{equation}
\label{eq4}
\triangle u = 1\quad \text{in}\quad D,\qquad u\big{|}_S = 0,
\quad u_N\big{|} = \mu=const>0,
\end{equation}
then $S$ is a sphere.}

This result is obtained by A. D. Alexandrov's "moving plane" argument, and
is equivalent to the following result:

{\it If 
\begin{equation}
\label{eq5}
\frac{1}{|D|} \int_D h(x)dx = \frac{1}{|S|}\int_S h(s)ds,\qquad \forall h\in \mathcal{H},
\end{equation}
then $S$ is a sphere.}

The equivalence of \eqref{eq4} and \eqref{eq5} can be proved as follows. 

Suppose
\eqref{eq4} holds. Multiply \eqref{eq4} by an arbitrary $h\in \mathcal{H}$, 
integrate by parts and get
\begin{equation}
\label{eq6}
\int_D h(x)dx = \mu \int_S h(s)ds.
\end{equation}
If $h=1$ in \eqref{eq6}, then one gets $\mu=\frac{|D|}{|S|}$, 
so \eqref{eq6} is identical to \eqref{eq5}.

Suppose \eqref{eq5} holds. Then \eqref{eq6} holds. Let $v$ solve the 
problem $\triangle v = 1$
in $D$, $v\big{|}_S=0$. This $v$ exists and is unique. Using
\eqref{eq6}, the equation $\triangle h = 0$ in $D$, and the Green's 
formula, one gets
\begin{equation}
\mu \int_S h(s)ds = \int_D h(x)dx = \int_D h(x)\triangle v dx = \int_S h(s)v_N ds.
\end{equation} 
Thus,
\begin{equation}
\label{eq8}
\int_S h(s)[v_N - \mu ]ds = 0,\qquad \forall h\in \mathcal{H}.
\end{equation}
The set of restrictions on $S$ of all harmonic functions in $D$ is dense 
in $L^2(S)$ (see, e.g., \cite{R512}).
Thus, \eqref{eq8} implies $v_N\big{|}_S = \mu$. Thus, \eqref{eq4} holds.

\section{Results and proofs}
Our main results are formulated in the following theorem

\begin{thm}
\label{thm1}
Let $D\subset \mathbb{R}^3$ be a bounded domain homeomorphic to a ball, $S$
be its Lipschitz boundary, $D':=\mathbb{R}^3\backslash D$. 
If any one of the following assumptions holds, then $S$ is a sphere:
\begin{enumerate}
\item
\begin{equation}
\label{eq9}
u(x):= \int_S \frac{ds}{4\pi |x-s|} = \frac{c}{|x|},\qquad c = const,\quad 
\forall x\in B'_R,
\end{equation}
where $B_R':= \{x: |x|>R\}$,\, $D\subset B_R$,\, $B_R:=\mathbb{R}^3\backslash B'_R$;
\item
\begin{equation}
\label{eq10}
\frac{1}{|S|}\int_S h(s)ds = h(0),\qquad \forall h\in \mathcal{H};
\end{equation}
\item
There exists a solution to the problem
\begin{equation}
\label{eq11}
\triangle_y u=\delta(y)\quad \text{in}\quad D,\quad u\big{|}_S = 0,\quad
u_N\big{|}_S = c_1 = const,
\end{equation}
where $\delta(y)$ is the delta-function.
\end{enumerate}

In \eqref{eq10} $0$ is the origin, $0\in D$, $|S|$ is the surface area of 
$S$, $\mathcal{H}$
is the set of all harmonic functions in a domain containing $D$. 
\end{thm}

\begin{proof}
\begin{enumerate}
\item
Assume \eqref{eq9}. Then $c=\frac{|S|}{4\pi}$ as one can see by taking 
$|x|\to\infty$.
If \eqref{eq9} holds for $\forall x\in B'_R$ then, by the unique 
continuation property for harmonic functions, \eqref{eq9} 
holds $\forall x\in D'$. Let $N_s$ be a unit normal to
$S$ at the point $s\in S$, pointing into $D'$. The known jump formula for 
the normal derivative of a single-layer potential (\cite[p.14]{R190}) yields
\begin{equation}
\label{eq12}
u^+_{N_{s_0}} = u^-_{N_{s_0}}+1,\qquad u_{N_{s_0}}^- = -\frac{|S|}{4\pi}\frac{N_{s_0}\cdot s_0}{|s_0|^3},\qquad s_0\in S, 
\end{equation}  
If $S$ is not a sphere, then there exists an $s_0\in S$, 
$|s_0|\le |s|$, $\forall s\in S$.
The ball $B_{|s_0|}$ of radius $|s_0|$, centered at the origin, belongs to 
$D$.
At the point  $s_0$ the normal $N_{s_0}$ to $S$ is directed along 
the vector $s_0$, so
\begin{equation}
\label{eq13}
u^-_{N_{s_0}} = - \frac{|S|}{4\pi |s_0|^2} < -1,
\end{equation} 
because $|S|> 4\pi |s_0|^2$ by the isoperimetric inequality (\cite{F}). 
This and formula \eqref{eq12} imply
\begin{equation}
\label{eq14}
u^+_{N_{s_0}}<0.
\end{equation} 
On the other hand,
\begin{equation}
u(s) = \frac{1}{4\pi |s|}\le \frac{1}{4\pi |s_0|}.
\end{equation}
So the harmonic and continuous in $D$ function $u(x)$ attains 
its maximum on $S$
at the point $s_0$, because $u\big{|}_S = \frac{1}{4\pi |s|}\big{|}_S$.
Therefore, by the maximum principle, 
$$u(x)\le u(s_0),\quad \forall x\in D.$$
In particular, $u(s_0)-u(s_0 - \epsilon N_{s_0})\ge 0$ for 
all sufficiently small $\epsilon>0$.
Consequently, $u_{N_{s_0}}\ge 0$. This contradicts \eqref{eq14}, and the contradiction
proves that $S$ is a sphere.

\item
Assume \eqref{eq10}. Let $h(y)=\frac{1}{4\pi |x-y|}, x\in D', y\in D$.
This function is a harmonic
function in $D$. Thus, \eqref{eq10} yields \eqref{eq9}:
\begin{equation}
\label{eq16}
\int_S \frac{ds}{4\pi |x-s|} = \frac{|S|}{4\pi |x|}, \qquad c:=\frac{|S|}{4\pi},\quad \forall x\in D'.
\end{equation}
We have already proved that \eqref{eq16} implies that $S$ is a sphere.
Therefore, the Assertion 2 of Theorem 1 is established.

\item
Assume \eqref{eq11}. Multiply \eqref{eq11} by $\frac{1}{4\pi|x-y|},\, x\in D'$,
integrate over $D$, and then integrate by parts to get
\begin{equation}
c_1\int_S \frac{ds}{4\pi |x-s|} = \frac{1}{4\pi |x|},\qquad \forall x\in D'.
\end{equation}
By the result, proved in Assertion 1, this implies that $S$ is a sphere.
Therefore, Assertion 3 of Theorem 1 is proved.
\end{enumerate}

\end{proof}

\end{document}